\documentclass[11pt]{amsart}

\usepackage[margin=1.2in]{geometry}

\usepackage{amssymb}
\usepackage{amscd}
\usepackage{verbatim}
\usepackage[latin1]{inputenc}
\usepackage{float}
\usepackage{bm}
\usepackage[T1]{fontenc}
\usepackage{enumitem}
\usepackage{xspace,xcolor}
\usepackage{tikz-cd}
\usepackage{caption}
\usepackage[labelformat=simple]{subcaption}

\usepackage[breaklinks,colorlinks,linkcolor=red,citecolor=blue]{hyperref}
\usepackage{cleveref}
\usepackage[alphabetic]{amsrefs}

\hypersetup{linktocpage} 

\DeclareMathAlphabet{\mathscr}{U}{BOONDOX-cal}{m}{n}
\SetMathAlphabet{\mathscr}{bold}{U}{BOONDOX-cal}{b}{n}
\DeclareMathAlphabet{\mathbscr}{U}{BOONDOX-cal}{b}{n}

\newcommand{\f}[1]{\mathbb{#1}}
\newcommand{\ns}[1]{\mathscr{#1}}
\newcommand{\ta}[1]{\widetilde{#1}}
\newcommand{\nc}[1]{\mathcal{#1}}
\newcommand{\p}[1]{\text{Spec}{#1}}

\newtheorem{thm}{Theorem}
\newtheorem{cor}[thm]{Corollary} 
\newtheorem{prop}[thm]{Proposition}
\newtheorem{lemma}[thm]{Lemma}

\theoremstyle{definition} 
\newtheorem{defn}[thm]{Definition}
\theoremstyle{remark} 
\newtheorem{ex}[thm]{Example}
\newtheorem{remark}[thm]{Remark}

\Crefname{thm}{Theorem}{Theorems}
\Crefname{cor}{Corollary}{Corollaries}
\Crefname{prop}{Proposition}{Propositions}
\Crefname{lemma}{Lemma}{Lemmas}
\Crefname{defn}{Definition}{Definitions}
\Crefname{remark}{Remark}{Remarks}
\Crefname{ex}{Example}{Examples}

\begin{document}

\title{Geometry over the tropical dual numbers}
\author[Keyvan Yaghmayi]{Keyvan Yaghmayi}
\address{Department of Mathematics, University of Utah, 155 S 1400 E, Salt Lake City, Utah 84112, USA}
\email{yaghmayi@math.utah.edu}
\thanks{Author's research is supported in part by NSF FRG grant DMS-1265285}
\date{\today}

\begin{abstract}
We introduce \textit{tropical dual numbers} $\ta{\f{T}}$ as an extension of tropical semiring $\f{T}$.  By this innovation, one can work with honest ideals, instead of congruences, and recover the Euclidean topology on affine tropical space $\f{T}^k$ similar to Zariski's approach in classical algebraic geometry.  The bend loci of an ideal over $\ta{\f{T}}$ coincides with the bend loci of a congruence over $\f{T}$ and this enables $\ta{\f{T}}$ to serve as an algebraic structure for tropical geometry.  Tropical Zariski topology on $\f{T}^k$ whose closed sets are non-linear loci of ideals $\ns{I}$ in $\ta{\f{T}}[x]$ offers an alternative point of view to strong Zariski topology defined by Giansiracusa.
\end{abstract}

\maketitle
\tableofcontents

\section{Introduction}
Tropical geometry, that might broadly described as algebraic geometry over the semifield $\f{T} := \left(\f{R} \cup \{\infty \}, \min, + \right)$ of tropical numbers, has been received a lot of attention in recent years.  One can define tropical varieties as non-linear locus of ideals and equip $\f{T}^k$ with the topology whose closed sets are tropical varieties. However, difficulties arise immediately: the subspace topology on tropical variety $V \subseteq \f{T}^k$ is not an intrinsic invariant of $V$ \textit{i.e.} depends on the embedding of $V$ into affine space $\f{T}^k$.  See \Cref{embedding}. \vspace{0.1cm}

Researchers addressed this issue in several different approaches.  In \cite{MS} and \cite{Mikh} they simply equip the affine space, \textit{e.g.} $\f{T}^k$ in our setting, with the Euclidean topology.  The authors in \cite{BE} and \cite{JM} work with congruences instead of ideals.  The idea of working with congruences in a semiring is that the pair $(f, g)$ represents the subtraction $f-g$.  Since in a ring one has $(f-g)(f'-g') = (ff'+gg')-(fg'+f'g)$, in order to make the pair $(f, g)$ compatible with multiplication of the semiring, one needs to define the product of pairs as $(f, g) \times (f', g') := (ff'+gg', fg'+f'g)$.  It is straightforward to show that the topology on $\f{T}^k$ whose closed sets are congruence varieties is the Euclidean topology (in this paper it follows from \Cref{etztopo,Enhanced}).  In \cite{IR} and \cite{Izh1} Izhakian and Rowen extend $\f{T}$ to \textit{supertropical semiring} which is not an idempotent semiring.  In addition, their topology is more complicated than the Euclidean topology.  See \cite{Izh2} Section 1.  Jeffrey and Noah Giansiracusa \cite{GG2} define strong Zariski topology by using the $\f{T}$-points of arbitrary subschemes as the closed sets and they show it gives the Euclidean topology on $\f{T}^k$. \vspace{0.1cm}

The purpose of this paper is to introduce the semiring of \textit{tropical dual numbers} $\ta{\f{T}}$, see \Cref{trop-dual-num}, such that the bend loci of ideals of $\ta{\f{T}}[x_1, \dots, x_k]$ induce the Euclidean topology on $\f{T}^k$.  Loosely speaking, one might think about the semiring $\ta{\f{T}}$ as $\f{T}$ without the idempotency.  Then the bend loci of polynomial $x^2 + x + x +1$ is $[0, 1]$, namely, all points $a \in \f{T}$ where $\min \{2x, x, x, 1\}$ achieves its minimum in at least two monomial terms. \vspace{0.1cm}

The semiring of tropical dual numbers $\ta{\f{T}}$ is an extension of $\f{T}$ by adjoining a nonzero nilpotent element $\epsilon$ such that $\epsilon^2 = 1_{\f{T}}$.  In other words, a tropical dual number is given by an expression of the form $a + b \epsilon$ where $a, b \in \f{T}$ and $\epsilon^2 = 1_{\f{T}}$.  Then the polynomial $x^2 + x + x +1$ will be presented, with coefficients in $\ta{\f{T}}$, as $x^2 + x + \epsilon x +1$, and in this way, we are allowing polynomial to have more terms. \vspace{0.1cm}

Another way is to think of $\ta{\f{T}}$ as an extension of $\f{T}$ in the same way that one extends $\f{R}$ to $\f{C}$: we divide the polynomial semiring $\f{T}[\epsilon]$ by congruence $\langle \epsilon^2 \sim 1_{\f{T}} \rangle$ and $\epsilon \in \ta{\f{T}}$ will play the role of the missing \textit{negative} $1_{\f{T}}$.  This point of view and its relation with congruences are explained with more details in \Cref{Semiring of Tropical Dual Numbers}. \vspace{0.1cm}

If $\ns{I}$ be an ideal in $\ta{\f{T}}[x]$, then, in \Cref{etztopo}, we shall show the topology on $\f{T}^k$ whose closed sets are $V(\ns{I})$ is the Euclidean topology, hence it agrees with the \textit{strong Zariski topology} on $\f{T}^k$, defined in \cite{GG2} Section 3.4, which is the one used to show that the universal tropicalization of an affine variety is homeomorphic to the inverse limit of all tropicalizations (\cite{GG2} Theorem 4.1.1) and hence to the Berkovich analytification of the variety \cite{SP}. \vspace{0.1cm}

In \Cref{Tropical Varieties and Congruence Varieties} we make connections among classical tropical varieties, congruence varieties, and tropical varieties over $\ta{\f{T}}$.  We show the horizontal embedding of congruence variety $V \subseteq \f{T}^k$ into $\f{T}^{k+1}$ is a tropical variety.  Similar assertion for a variety over $\ta{\f{T}}$ is proved in \Cref{Enhancedclassical}. \vspace{0.3cm}

\textbf{Acknowledgments}. My sincere thanks go to my advisor, Professor Tommaso de Fernex, for his guidance, encouragement and support during the development of this paper.  I would like to thank Aaron Bertram for his useful comments and discussions.  I want to express my sincere gratitude to Jeffrey and Noah Giansiracusa for their many helpful comments.  I am also grateful for conversation and suggestions from Diane Maclagan.  

\section{Semiring of Tropical Dual Numbers} \label{Semiring of Tropical Dual Numbers}

Let $\f{T}:= \left(\f{R} \cup \{\infty \}, \min, + \right)$ be the \textit{tropical semifield}, that is the set $\f{R} \cup \{\infty \}$ equipped with the following two operations: the minimum called the tropical addition with neutral element $0_{\f{T}}=\infty$ and the addition called the tropical multiplication with neutral element $1_{\f{T}}=0$.  Let $\f{T}[x]=\f{T}[x_1, x_2, ..., x_k]$ be the tropical polynomial semiring in $k$ variables.  We write $x^n$ for monomial $x_1^{n_1}x_2^{n_2}...x_k^{n_k}$ in $\f{T}[x]$ and, in order to simplify the notation, we assume $x_i^0=1_{\f{T}}=0$ which allows us to write $x^n=x_1^{n_1}x_2^{n_2}...x_k^{n_k}$ for $x_{i_1}^{n_{i_1}}x_{i_2}^{n_{i_2}}...x_{i_r}^{n_{i_r}}$ where some of $n_i$'s are zero.  A tropical polynomial is a finite sum 
\begin{align}
f(x)=\sum_n {c_nx^n} \label{tropoly}
\end{align}
where $c_n \in \f{R}$ and $n=(n_1, n_2, ..., n_k) \in \f{Z}_{\ge 0}^k$.  Since distinct tropical polynomials could define the same polynomial function we write $f(x) = \min_n  \{ c_n + \sum_i n_ix_i \}$ when we consider $f$ as a function.  Its value at $a = (a_1, a_2, ..., a_k) \in \f{T}^k$ is 
\begin{align}
f(a) = \min_n  \{ c_n + \sum_i n_ia_i \} \label{tropolyvalue}
\end{align} 

When there is no ambiguity in the context, we might use "$+$" and "$\sum$" in the sense of tropical or classical summation.  For instance, the $\sum$ in \ref{tropoly} is the tropical summation but the $\sum$ in \ref{tropolyvalue} is the classical addition.
  
\begin{defn} \label{ideal}  
An ideal $I$ of $\f{T}[x]$ is a submonoid of $\left(\f{T}[x], \textit{min}\right)$ which is closed under multiplication by elements of $\f{T}[x]$.  The \textit{tropical variety} associated to ideal $I$ is $V(I)=\bigcap_{f \in I}  V(f)$ where $V(f)$ is the nonlinear locus of $f: \f{T}^k \to \f{T}$.
\end{defn}

\begin{defn}\label{con}
A congruence $E$ of $\f{T}[x]$ is a subset of $\f{T}[x] \times \f{T}[x]$ that satisfies the following properties:
\begin{enumerate}[label=$(\alph*)$]
\item For every $f \in \f{T}[x]$ one has $(f, f) \in E$
\item $(f, g) \in E$ if and only if $(g, f) \in E$
\item If $(f, g) \in E$ and $(g, h) \in E$ then $(f, h) \in E$
\item If $(f, g) \in E$ and $(f', g') \in E$ then $(f+f', \,g+g') \in E$
\item If $(f, g) \in E$ and $(f', g') \in E$ then $(ff', \,gg') \in E$
\end{enumerate}
\end{defn}

In fact, a congruence of $\f{T}[x]$ is an equivalence relation on $\f{T}[x]$ that respects the semiring structure of $\f{T}[x]$.  In the category of rings there is a one-to-one correspondence between ideals and congruences: given an ideal $I$ of a ring $R$ let $E_I=\{ (a, b) \in R \times R \colon a-b \in I\}$.  Then $E_I$ is a congruence of $R$. Conversely, for congruence $E$ of $R$ let $I_E=\{a \colon (a, 0) \in E \}$.  One can easily show that $I_E$ is an ideal of $R$. For semirings there is not such a bijective correspondence anymore simply because there is no subtraction.  Here we give an explicit example

\begin{ex} \label{F1field}
It is straightforward to check that the tropical semiring $\f{T}$ has no nontrivial proper ideal.  On the other hand, $\f{T}$ has three congruences 
\begin{enumerate}[label=-]
\item The trivial congruence $\Delta = \{ (a, a) \colon a \in \f{T} \}$
\item The improper congruence $\f{T} \times \f{T}$
\item The nontrivial proper congruence $E_{\infty} = \f{T} \times \f{T} \setminus \{ (a, \infty) , (\infty, a) \colon a \neq \infty \}$. 
\end{enumerate}
To see this, let $E$ be a nontrivial proper congruence of $\f{T}$.  Since $E$ is nontrivial, there is $a < b$ such that $(a, b) \in E$.  We claim $b \neq \infty$ because otherwise $(a, \infty) \in E$ and then for any $c \in \f{T}$ we have $(c, c) + (a, \infty) = (a+c, \infty) \in E$.  This implies $(d, \infty) \in E$ for any $d$ which contradicts properness.  Thus, $(a, b) \in E$ for some $a < b < \infty$.  Therefore, $(-a, -a) + (a, b) = (0, b-a) \in E$.  Given a positive real number $c \in \f{R}$, pick a large positive integer $m$ such that $c < m(b-a)$, then $(\min\{0, c\}, \min\{c, m(b-a)\})=(0, c) \in E$.  Finally $(0, c)+(-c, -c)=(-c, 0) \in E$.  We conclude $(0, c) \in E$ for any real number $c$ and $E = E_{\infty}$. 
\end{ex}

If $E$ is a congruence of $\f{T}[x]$, then the \textit{quotient semiring} $\f{T}[x] / E$ is the set of all equivalence classes with natural addition and multiplication, \textit{i.e.}, $[f]+[g] = [f+g]$ and $[f].[g] = [fg]$ where $[f]$ is the equivalence class of $f \in \f{T}[x]$.  In particular, in \Cref{F1field}, the quotient of $\f{T}$ by $E_{\infty}$ has two elements: $[0_{\f{T}}]$ and $[1_{\f{T}}]$ with usual addition and multiplication except that $[1_{\f{T}}] + [1_{\f{T}}] = [1_{\f{T}}]$. \vspace{0.2cm}

The \textit{twisted product} of two pairs $(f, g)$ and $(f', g')$ considered in \cite{BE} is
\begin{align*}
(f, g) \times (f', g') &= \left( ff'+gg', \, fg'+gf' \right) 
\end{align*} 
One can show without difficulty that in \Cref{con} property $(e)$ could be replaced with the following condition:
\begin{enumerate}[label=$(e')$]
\item If $(f, g) \in E$ and $(f', g') \in \f{T}[x] \times \f{T}[x]$ then $(f, g) \times (f', g') \in E$
\end{enumerate} \vspace{0.1cm}

If $E$ be a congruence and $(f, g) \in E$, then it is convenient to write $f \sim g$.  We write $E = \langle f_\alpha \sim g_\alpha \rangle_{\alpha \in \Lambda}$ for the smallest congruence that contains relations $\{ f_\alpha \sim g_\alpha \}_{\alpha \in \Lambda}$ and we call it the congruence generated by (equivalence) relations $f_\alpha \sim g_\alpha$ where $\alpha \in \Lambda$.
  
\begin{defn}
The \textit{congruence variety} associated to $E$ is
\begin{align*}
V(E) = \left\{ a \in \f{T}^k : f(a) = g(a) \ \, \text{for every} \ f \sim g \ \text{in} \ E \right\}
\end{align*}
\end{defn}

\begin{prop}[\cite{BE}, Proposition 3.1]
The congruence variety of $E=\langle f_1 \sim g_1, \ldots, f_r \sim g_r \rangle$ is
\begin{align*}
V(E)=\bigcap_{i=1}^r  \left\{ a \in \f{T}^k : f_i(a) = g_i(a) \right\}
\end{align*}
\end{prop}

Now we introduce the semiring of tropical dual numbers that provides an adaptable algebraic structure for tropical geometry.  The "thought process" is somehow similar to constructing the complex numbers from the real numbers \footnote{This is just a loose idea since after extending $\f{T}$ to $\ta{\f{T}}$ the only element with additive inverse in $\ta{\f{T}}$ is still $0_{\ta{\f{T}}}=\infty$.}:  Since $\f{T}$ is suffering from lack of subtraction we try to add "negative $1_{\f{T}}$", which we represent with a new symbol $\epsilon$,  to $\f{T}$ by taking the quotient of $\f{T}[\epsilon]$ by congruence $\langle \epsilon^2 \sim 1_{\f{T}} \rangle$.  Then the pair $(a, b)$, \textit{i.e.} the imaginary subtraction $a-b$, will be represented by $a+b\epsilon$ in $\ta{\f{T}}$, and moreover, in the quotient semiring, the product $(a+b\epsilon)(a'+b'\epsilon)=(aa'+bb')+(ab'+ba')\epsilon$ will be compatible with the twisted product of pairs.  In this way $\f{T}$ extends to a bigger semiring that makes us needless of congruences (See \Cref{polypro,Enhanced}).

\begin{defn} \label{trop-dual-num}
The semiring of \textit{tropical dual numbers} $\ta{\f{T}}$ is the quotient of $\f{T}[\epsilon]$ by the congruence $E = \langle \epsilon^2 \sim 1_{\f{T}} \rangle$.
\end{defn}

Elements of $\ta{\f{T}}$ can be written, in a unique way, as $a+b\epsilon$ where $a, b \in \f{T}$.  For instance, $0_{\ta{\f{T}}} = 0_{\f{T}} + 0_{\f{T}}\epsilon = \infty + \infty \epsilon = \infty$ is the additive identity of $\ta{\f{T}}$ and $1_{\ta{\f{T}}} = 1_{\f{T}} + 0_{\f{T}}\epsilon = 1_{\f{T}}$ is the multiplicative identity.  In the following, we denote elements of $\ta{\f{T}}$ by Greek letters $\alpha$, $\beta$, $\gamma$, $\dots$ and if we need to work with the components of $\alpha$ then we write $\alpha=a+b\epsilon$.  The product of $\alpha=a+b\epsilon$ and $\beta=c+d\epsilon$ is $(a+b\epsilon)(c+d\epsilon) = (ac+bd)+(ad+bc)\epsilon$.  Obviously $\ta{\f{T}}$ is idempotent since it is the quotient of an idempotent semiring. \vspace{0.1cm}

Note that $\f{T}$ is a sub-semiring of $\ta{\f{T}}$ consisting of all elements $a+b\epsilon$ where $b=\infty$.  It is easy to see that $\ta{\f{T}}$ is a domain but not a semifield.  In fact, for $\alpha=a+b\epsilon$ let $-\alpha:=-a+(-b)\epsilon$.  Then $\alpha=a+b\epsilon$ has multiplicative inverse if and only if either $a=\infty$ or $b=\infty$ but not both, and in this case the inverse of $\alpha$ is $-\alpha$. \vspace{0.1cm}

The polynomial $\ns{f}$ defines a function $\ns{f} : \ta{\f{T}}^k \to \ta{\f{T}}$ but it has no "$\min$, $+$" interpretation since, unfortunately, the semiring of tropical dual numbers is not totally ordered.  One can put an order on $\ta{\f{T}}$ to obtain a geometric interpretation for the graph of $\ns{f}$ and then define $\ns{V}(\ns{f}) \subseteq \ta{\f{T}}^k$ similar to \Cref{ideal}, however, in the following we shall define tropical varieties $V(\ns{f})$ as subsets of $\f{T}^k$. \vspace{0.1cm}

Let $\alpha=a+b\epsilon$ be an element of $\ta{\f{T}}$.  Thinking of $\alpha$ as a degree one polynomial in $\f{T}[\epsilon]$, it defines a function $\alpha : \f{T} \to \f{T}$ such as $\alpha(c) = \min\{a, b+c\}$.  Now consider
\begin{equation*}
\begin{split}
\pi \colon  \ta{\f{T}} & \longrightarrow \f{T} \\
\alpha & \longmapsto \alpha(1_{\f{T}}) = \min\{a, b\}
\end{split}
\label{nu}
\end{equation*} 

It is straightforward to check that $\pi$ is an idempotent semiring homomorphism and $\pi \rvert_{\f{T}}=\text{id}_{\f{T}}$.  By applying $\pi$ coordinatewise, we have a map from $\ta{\f{T}}^k$ onto $\f{T}^k$, moreover, by applying $\pi$ to the coefficients of $\ns{f}$ we obtain a semiring homomorphism $\ta{\f{T}}[x] \to \f{T}[x]$.  We denote these maps by the same $\pi$.  For $\ns{f}(x)=\sum_n \alpha_nx^n \in \ta{\f{T}}[x]$ one has $\pi(\ns{f})(x) = \sum_n \pi(\alpha_n)x^n$ and $\pi(\ns{f}(\alpha)) = \pi(\ns{f})(\pi(\alpha))$ where $\alpha \in \ta{\f{T}}^k$.  In particular, for $a \in \f{T}$ we have $\pi(\ns{f}(a)) = \ns{f}(a)|_{\epsilon = 1_{\f{T}}}$. \vspace{0.2cm}

Let $\ns{f} \in \ta{\f{T}}[x]$.  If one considers $\ta{\f{T}}[x]$ as a $\ta{\f{T}}$-algebra, then a monomial term of $\ns{f}$ is of the form $\alpha x^n = (a_n+b_n\epsilon)x^n$, however, $\ta{\f{T}}[x] = \f{T}[x, \epsilon]$ is also a $\f{T}$-algebra where $\epsilon$ appears with degree at most $1$.  

\begin{defn}\label{nhctropvar}  
A \textit{simple monomial terms} of $\ns{f} \in \ta{\f{T}}[x]$ is a monomial term in the $\f{T}$-algebra $\f{T}[x,\epsilon]$, that is, a monomial terms of the form either $cx^n$ or $cx^n\epsilon$ where $c \in \f{T}$.  The tropical variety $V(\ns{f})$ is the collection of all points $a \in \f{T}^k$ where $\ns{f}(a)|_{\epsilon = 1_{\f{T}}}$ achieves its value in at least two simple monomial terms.  For ideal $\ns{I}$ in $\ta{\f{T}}[x]$, the associated tropical variety is $V(\ns{I}) = \bigcap_{\ns{f} \in \ns{I}} V(\ns{f})$. 
\end{defn}

In other words, if we denote a simple monomial term with $c_nx^n\epsilon^{\delta_n}$ where $c_n \in \f{T}$ and $\delta_n \in \{0, 1\}$, then $V(\ns{f})$ consists of those points $a \in \f{T}^k$ for which there are at least two simple monomial terms $c_mx^m\epsilon^{\delta_m}$ and $c_nx^n\epsilon^{\delta_n}$ such that $c_ma^m=c_na^n \leq \pi(\ns{f}(a))$.  Here the indices $m$ and $n$ may or may not be equal, for instance, $V( 1 + x + \epsilon x) = (-\infty, 1]$ because for $a \leq 1$ we have
\begin{align*}
a = \epsilon a |_{\epsilon = 1_{\f{T}}} \leq \ns{f}(a) |_{\epsilon = 1_{\f{T}}}
\end{align*}

\begin{ex} \label{xmplen}
Let $\ns{f}(x)=(3+1\epsilon)x^2 + (1+1\epsilon)x + 2\epsilon$.  Then for $a \in \f{T}$, we have $\ns{f}(a)=(3a^2+1a)+(1a^2+1a+2)\epsilon$ and $\pi(\ns{f}(a))=\min \{2a+1, a+1, 2\}$.  Note that $V(\ns{f})=[0, 1]$ since simple monomial terms in $\ns{f}$ are $\{3x^2, 1x^2\epsilon, 1x, 1x\epsilon, 2x^0\epsilon \}$ and for $a \in [0, 1]$ one has $\pi(1a)=\pi(1a\epsilon)=1+a$. 
\end{ex}

\Cref{nhctropvar,ideal} are consistent in the following sense.  The proof is postponed to the end this section. 

\begin{prop} \label{coincide}
Let $I$ be an ideal of $\f{T}[x]$ and let $\ns{I} = \langle I \rangle$ be the ideal in $\ta{\f{T}}[x]$ which is generated by $I$.  Then $V(\ns{I}) = V(I)$.
\end{prop}

The following proposition follows immediately from \Cref{nhctropvar}.

\begin{prop} \label{newinc}
Let $\ns{I}$ and $\ns{J}$ be ideals in $\ta{\f{T}}[x]$.  If $\ns{I} \subseteq \ns{J}$, then $V(\ns{I}) \supseteq V\left(\ns{J}\right)$.
\end{prop}

\begin{remark} \label{rmrken} 
Fix a polynomial $f(x)=\sum_n {c_nx^{m_n}}$ in $\f{T}[x]$.  The graph of $z = f(x)$ is the lower envelope of the union of hyper-planes $\{z = c_n+m_nx\} \subset \f{T}^{k+1}$.  The projection of this graph onto $\f{T}^k$ induces a finite decomposition of $\f{T}^k$ into closed $k$-dimensional polyhedra.  Let's write $\f{T}^k = \cup_{j=1}^rP_j$ where the interior of the polyhedron $P_j$ is exactly those $a \in \f{T}^k$ where for some monomial term like $c_jx^{m_j}$ we have $c_j+m_ja < \min_{n \neq j} \{c_n+m_na \}$.  The boundary of these polyhedra, where at least two monomials attain the minimum in $\min_n \{c_n+m_na \}$, is $V(f)$ (for more details see \cite{MS} Section 3.1).  Now let $S$ be a subset of $\{1, 2, ..., r \}$ and consider polynomial $\ns{f}$ in $\ta{\f{T}}[x]$ obtained from $f$ by replacing $c_jx^{m_j}$ with $(c_j+c_j\epsilon)x^{m_j}$.  Then $V(\ns{f})$ will be the boundary of $P_j$'s plus the interior of those polyhedra $P_j$ where $j \in S$. In \Cref{enhancedvstrop} we illustrated this with an explicit example. 
\end{remark}

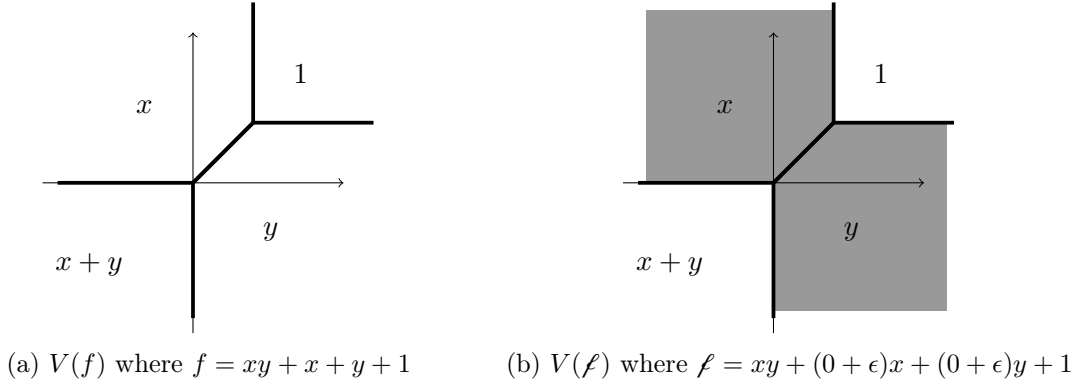
\begin{figure}[H]
\captionsetup{font=footnotesize}
\centering
\subcaptionbox{$V(f)$ where $f = xy + x + y + 1$}[.49\textwidth]{\begin{tikzpicture}
% Here we draw left picture:
\draw [->] (-2,0) -- (2,0);
\draw [->] (0,-2) -- (0,2);
\draw (1.2,1.2) node [above right]{$1$};
\draw (-0.4,0.8) node [above left]{$x$};
\draw (0.8,-0.4) node [below right]{$y$};
\draw (-0.8,-0.8) node [below left]{$x+y$};
\draw [line width=0.5mm](-1.8,0) -- (0,0) -- (0.8,0.8) -- (0.8,2.4);
\draw [line width=0.5mm](0,-1.8) -- (0,0) -- (0.8,0.8) -- (2.4,0.8);
\end{tikzpicture}}
\subcaptionbox{$V(\ns{f})$ where $\ns{f} = xy + (0 +\epsilon)x + (0 +\epsilon)y + 1$}[.49\textwidth]{\begin{tikzpicture}
% Here we draw right picture:
%\fill[gray] (-1.7, -1.7) rectangle (0, 0);
\fill[gray!80] (-1.7, 2.3) rectangle (0, 0);
\fill[gray!80] (0.8,0.8) rectangle (2.3,-1.7);
\fill[gray!80] (0,2.3) rectangle (0.8,-1.7);
%\fill[gray] (0.8,0.8) rectangle (2.3, 2.3);
\draw [->] (-2,0) -- (2,0);
\draw [->] (0,-2) -- (0,2);
\draw (1.2,1.2) node [above right]{$1$};
\draw (-0.4,0.8) node [above left]{$x$};
\draw (0.8,-0.4) node [below right]{$y$};
\draw (-0.8,-0.8) node [below left]{$x+y$};
\draw [line width=0.5mm](-1.8,0) -- (0,0) -- (0.8,0.8) -- (0.8,2.4);
\draw [line width=0.5mm](0,-1.8) -- (0,0) -- (0.8,0.8) -- (2.4,0.8);
\end{tikzpicture}}
\caption{A tropical variety over $\ta{\f{T}}$ looks like a classical tropical variety that some of the regions are filled in.}
\label{enhancedvstrop}
\end{figure}

Before proceeding, we would like to adopt the notations of \Cref{rmrken}.  We use script fonts to denote polynomial $\ns{f}$ and ideal $\ns{I}$ over the semiring of tropical dual numbers $\ta{\f{T}}$. The polynomial $f$ and ideal $I$ live in $\f{T}[x]$.  According to \Cref{nhctropvar,ideal} both (classical) tropical variety $V(I)$ and our generalized tropical variety $V(\ns{I})$ are subsets of $\f{T}^k$.  \vspace{0.08in}

The following properties of tropical varieties are well-known:

\begin{prop} \label{knownprop}
\begin{enumerate}[label=$(\alph*)$]
\item Let $I_1, \dots, I_r$ be ideals in $\f{T}[x]$.  We have $V(I_1) \cup \, \dots \, \cup V(I_r) = V(I_1\dots I_r)$. \label{knownunion} \item If $\{I_i\}_{i \in \Lambda}$ be a family of ideals in $\f{T}[x]$, then $\bigcap_{i \in \Lambda} V(I_i) = V \left(\sum_{i \in \Lambda} I_i \right)$.  \label{knownintersection}
\end{enumerate}
\end{prop}

Here we prove these properties for ideals over the semiring of tropical dual numbers. Our proof can be used to give a direct proof of \Cref{knownprop}.  

\begin{prop} \label{newprop}
\begin{enumerate}[label=$(\alph*)$]
\item For ideals $\ns{I}_1, \dots, \ns{I}_r$ in $\ta{\f{T}}[x]$ one has $V(\ns{I}_1) \cup \, \dots \, \cup V(\ns{I}_r) = V(\ns{I}_1\dots\ns{I}_r)$.  \label{newunion} 
\item Let $\{ \ns{I}_i\}_{i \in \Lambda}$ be a family of ideals in $\ta{\f{T}}[x]$.  Then $\bigcap_{i \in \Lambda} V(\ns{I}_i) = V\left(\sum_{i \in \Lambda}\ns{I}_i\right)$. \label{newintersection}
\end{enumerate}
\end{prop}

\begin{proof} 
\begin{enumerate}[label=$(\alph*)$]
\item By induction, it is enough to show $V(\ns{I}) \cup V(\ns{J}) = V(\ns{I}\ns{J})$ where $\ns{I}$ and $\ns{J}$ are two ideals in $\ta{\f{T}}[x]$.  First, we show for polynomials $\ns{f}$ and $\ns{g}$ in $\ta{\f{T}}[x]$ one has $V(\ns{f}) \cup V(\ns{g}) = V(\ns{f}\ns{g})$.  Let $\ns{f}(x)=\sum_n {(a_n+b_n\epsilon)x^n}$ and $\ns{g}(x)=\sum_n {(a'_n+b'_n\epsilon)x^n}$ and assume $a \in V(\ns{f})$.  Then, with notations as above, $\ns{f}(a) |_{\epsilon = 1_{\f{T}}}$ attains the minimum in at least two simple monomials, say $c_rx^r\epsilon^{\delta_r}$ and $c_sx^s\epsilon^{\delta_s}$.  If $\ns{g}(a)= c'_ta^t$ then $(\ns{f}\ns{g})(a) |_{\epsilon = 1_{\f{T}}}$ will achieve its minimum in terms $c_rc'_ta^{r+t}$ and $c_sc'_ta^{s+t}$.  This shows $V(\ns{f}) \subseteq V(\ns{f}\ns{g})$.  By symmetry $V(\ns{g}) \subseteq V(\ns{f}\ns{g})$; therefore, $V(\ns{f}) \cup V(\ns{g}) \subseteq V(\ns{f}\ns{g})$. Conversely, if $a \notin V(\ns{f}) \cup V(\ns{g})$, then $\ns{f}(a)$ and $\ns{g}(a)$ will achieve their minimum in exactly one simple monomial term, let's say $\ns{f}(a) |_{\epsilon = 1_{\f{T}}} = c_ra^r$ and $\ns{g}(a) |_{\epsilon = 1_{\f{T}}} = c'_ta^t$.  Hence, $(\ns{f}\ns{g})(a) |_{\epsilon = 1_{\f{T}}}$ attains the minimum in $c_rc'_ta^{r+t}$ and $a \notin V(\ns{f}\ns{g})$. \vspace{0.1cm}

For the general case, note that $\ns{I}\ns{J}$ is generated by $\ns{h}=\ns{f}\ns{g}$ where $\ns{f} \in \ns{I}$ and $\ns{g} \in \ns{J}$.  Now if $a \in V(\ns{I}) = \bigcap_{\ns{f} \in \ns{I}}V(\ns{f})$ then, by above argument for the polynomial case, for every $\ns{h}=\ns{f}\ns{g}$ one has $a \in V(\ns{h})$.  This implies $V(\ns{I}) \subseteq V(\ns{I}\ns{J})$ and by symmetry $V(\ns{J}) \subseteq V(\ns{I}\ns{J})$; hence $V(\ns{I}) \cup V(\ns{J}) \subseteq V(\ns{I}\ns{J})$.  Conversely, assume $a \notin V(\ns{I}) \cup V(\ns{J})$.  Then $a \notin V(\ns{f})$ for some $\ns{f} \in \ns{I}$ and $a \notin V(\ns{g})$ for some $\ns{g} \in \ns{J}$ and therefore $a \notin V(\ns{f}\ns{g})$ where $\ns{f}\ns{g} \in \ns{I}\ns{J}$.  Thus, $a \notin V(\ns{I}\ns{J})$. \vspace{0.1cm}

\item If $a \in V \left(\sum_{i \in \Lambda}\ns{I}_i \right)$ then $a \in V(\ns{h})$ for every $\ns{h} \in \sum_{i \in \Lambda}\ns{I}_i$, in particular, $a \in V(\ns{f})$ for every $\ns{f} \in \ns{I}_i$ (where $i \in \Lambda$ is arbitrary).  Hence, $a \in V(\ns{I}_i)$ and this implies $a \in \bigcap_{i \in \Lambda} V(\ns{I}_i)$.  Conversely, if $a \notin V \left( \sum_{i \in \Lambda}\ns{I}_i \right)$ then $a \notin V(\ns{h})$ for some $\ns{h} \in \sum_{i \in \Lambda}\ns{I}_i$.  One can write $\ns{h}$ as finite sum $\ns{h}=\sum_{j=1}^m \ns{f}_{i_j}$ where $\ns{f}_{i_j} \in \ns{I}_{i_j}$.  Since $a \notin V(\ns{h})$, we know $\ns{h}(a) |_{\epsilon = 1_{\f{T}}}$ attains its value in exactly one simple monomial term, for instance in term $c_rx^r\epsilon^{\delta_r}$.  Without loss of generality, one can assume $c_rx^r\epsilon^{\delta_r}$ belongs to $\ns{f}_{i_1}$, and hence, $\ns{f}_{i_1}(a) |_{\epsilon = 1_{\f{T}}}$ attains its value in just a single simple monomial term.  Therefore, $a \notin V(\ns{f}_{i_1})$ and $a \notin \bigcap_{i \in \Lambda} V(\ns{I}_i)$. 
\end{enumerate} 
\end{proof}

\begin{cor} \label{cor1}
Let $\ns{f} \in \ta{\f{T}}[x]$ and let $\ns{I} = \langle \ns{f} \rangle$.  Then $V(\ns{I}) = V(\ns{f})$.
\end{cor}
\begin{proof}
By \Cref{newinc} we have $V(\ns{I}) \subseteq V(\ns{f})$. On the other hand, each $\ns{g} \in \ns{I}$ is in the form of $\ns{g} = \ns{f}\ns{h}$ for some $\ns{h} \in \ta{\f{T}}[x]$.  By \Cref{newprop}\ref{newunion} $V(\ns{g}) = V(\ns{f}) \cup V(\ns{h})$.  So, for every $\ns{g} \in \ns{I}$ we have $V(\ns{g}) \supseteq V(\ns{f})$ and this implies $V(\ns{I}) \supseteq V(\ns{f})$.
\end{proof}

\begin{remark}
In view of the \Cref{coincide}, one can see \Cref{knownprop} as a special case of \Cref{newprop}.  Note, however, that we are using \Cref{knownprop} to prove it. 
\end{remark}

\begin{proof}[proof of \Cref{coincide}]
First, we show that for ideal $\ns{I}$ in $\ta{\f{T}}[x]$ with generators $\{ \ns{f}_i \}_{i \in \Lambda}$ one has $V(\ns{I}) = \bigcap_{i \in \Lambda} V(\ns{f}_i)$.  Let $\ns{I}_i = \langle \ns{f}_i \rangle$ and note that $\ns{I} = \sum_{i \in \Lambda} \ns{I}_i$.  By \Cref{newprop}\ref{newintersection} and \Cref{cor1}
\begin{align*}
V(\ns{I}) = \bigcap_{i \in \Lambda} V(\ns{I}_i) = \bigcap_{i \in \Lambda} V(\ns{f}_i)
\end{align*}
  
Now, pick a family of generators $\{ f_i \}_{i \in \Lambda}$ for $I$.  Mind that for polynomial $f \in \f{T}[x]$ the \Cref{nhctropvar} for $V(f)$ coincides with the classical tropical variety $V(f)$ (see \Cref{ideal}).  Because $\{ f_i \}_{i \in \Lambda}$ is also a generator for $\ns{I}$ so we can write 
\begin{align*}
V(\ns{I}) = \bigcap_{i \in \Lambda} V(f_i) = V(I)
\end{align*}
\end{proof}

\section{Tropical Zariski Topology} \label{Tropical Zariski Topology} 

One can follow the ideas from algebraic geometry and define \textit{tropical topology} on $\f{T}^k$ by defining the closed sets to be the tropical varieties.  More precisely, $X \subseteq \f{T}^k$ is closed if and only if $X = \f{T}^k$ or $X=V(I)$ for some ideal $I$ in $\f{T}[x]$.  The ideal $I$, in contrast to classical algebraic geometry, is not necessarily finitely generated.  In other words, the \textit{Hilbert Basis Theorem} for semiring $\f{T}[x]$ fails in the sense that some ideals of $\f{T}[x]$ are not finitely generated. 
 
\begin{ex} \label{infgen}
The ideal $I = \langle x^n +1 : n \geq 1 \rangle$ in $\f{T}[x]$ is not finitely generated.  To see this, consider the ideal $J = \langle x^{n_1} +1, x^{n_2}+1, ..., x^{n_r}+1 \rangle$ and let $N > \max_i \{ n_i \}$.  Elements of $J$ are of the form 
\begin{align}
p_1(x^{n_1} +1) + \dots + p_r(x^{n_r} +1) \label{summ}
\end{align}
where $p_i$'s belong to $\f{T}[x]$.  Since monomial terms in \eqref{summ} do not cancel each other, it could be equal to $x^N + 1$ only if all $p_i$'s are $0_{\f{T}}$ but one, say $p_1$, which should be a monomial: $p_1 = ax^m$.  Then $ax^m(x^{n_1} +1)$ can not be equal to $x^N + 1$, \textit{i.e.}, $x^N + 1$ does not belong to $J$. 
\end{ex}

To verify that $\f{T}^k$ alongside subsets of the form $V(I)$ define a topology on $\f{T}^k$, note that the whole space $\f{T}^k$ and $V(1_{\f{T}}) = \emptyset$ are closed.  Furthermore, by \Cref{knownprop}\ref{knownunion} finite union of closed sets is closed and by \Cref{knownprop}\ref{knownintersection} arbitrary intersection of closed sets is closed.  One hurdle in working with this topology that shows up immediately is that the induced topology on tropical variety $V(I) \subseteq \f{T}^k$ depends on the embedding.  This is illustrated by the following simple example. 

\begin{ex} \label{embedding}
Let $V$ be the affine line $\f{T}^1$ and consider the embedding $i : \f{T}^1 \hookrightarrow \f{T}^2$ such that $i(a)=(a, 0)$, namely, $i(\f{T}^1) = V(x_2 + 0) \subset \f{T}^2$.  Then, in the induced topology, the line segment $\{ (a, 0) : a \in [0, 1] \}$ is a closed subset of $V(x_2 + 0)$ because it is cut out by $V(x_1^2 +x_1x_2 + x_1 + 1)$ (see \Cref{embedded}).  On the other hand, if one embeds the affine line into itself via the identity map, then $[0, 1] \subset \f{T}^1$ can not be a closed subset because closed subsets of $\f{T}^1$ are its finite subsets. 
\end{ex} 
\begin{figure}[H]
\captionsetup{font=footnotesize}
\begin{center}
\begin{tikzpicture}
\draw [->] (-2,0) -- (2.3,0);
\draw [->] (0,-1.5) -- (0,2);
%\draw [ultra thick](-1.8,0) -- (2.1,0);
\draw [line width=0.5mm](-1.3,-1.3) -- (0,0) -- (1,0) -- (1, 1.5);
\draw [line width=0.5mm](0, 1.5) -- (0,0) -- (1,0) -- (2.3, -1.3);
\draw (0.8,-0.7) node [below]{$x_1+x_2$};
\draw (2,0.5) node {$1$};
\draw (0.5,0.5) node [above]{$x_1$};
\draw (-0.8,0.5) node [left]{$2x_1$};
\end{tikzpicture}
\caption{If one embeds $\f{T}^1$ in $\f{T}^2$ via $i(a)=(a, 0)$, then the line segment $[0, 1]$ is a closed in $\f{T}^1$ since it is cut out by $V(x_1^2 +x_1x_2 + x_1 + 1)$.}
\label{embedded}
\end{center}
\end{figure}
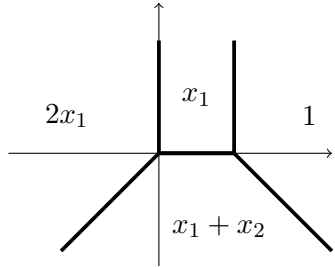

We will address this issue by working over tropical dual numbers, in other words, the \textit{tropical Zariski topology} on $\f{T}^k$ is the topology whose closed sets are $V(\ns{I})$ where $\ns{I}$ is an ideal in $\ta{\f{T}}[x]$.  This is in fact a topology because
\begin{enumerate}[label=$(\roman*)$]
\item The empty set $V(1_{\ta{\f{T}}}) = \emptyset$ and the whole space $V(x+\epsilon x)= \f{T}^k$ are closed.
\item By \Cref{newprop} finite union and arbitrary intersection of closed subsets are again closed. 
\end{enumerate} 
Moreover, any subset $X$ of $\f{T}^k$ will be equipped with the topology induced from the tropical Zariski topology on $\f{T}^k$, which will be called the \textit{tropical Zariski topology on $X$}.  In particular, a tropical variety $V$ is equipped with the this topology such that a subset $W$ of $V$ is closed in $V$ if and only if $W = V \cap W'$ for some closed set $W'$ in $\f{T}^k$. \vspace{0.1cm}
 
Our next goal in this section is to show that our tropical Zariski topology and the standard (Euclidean) topology on $\f{T}^k$ are the same.  To be precise, by the standard topology on $\f{T}$ we mean the topology with basis of open intervals either of the form $(a, b)$ or of the form $(c, \infty]$ where $a, b, c \in \f{R}$.  Then, the standard topology on $\f{T}^k$ will be the product topology, namely, it has basis of opens $\nc{B} = \{ \prod_{i=1}^k B_i \}$ where each $B_i \subseteq \f{T}$ is either of the form $(a_i, b_i)$ or of the form $(c_i, \infty]$ where $a_i, b_i, c_i \in \f{R}$.  

\begin{thm} \label{etztopo}
The tropical Zariski topology on $\f{T}^k$ is equal to the standard topology.  
\end{thm}
\begin{proof}
To see that each open basis $B = \prod_{i=1}^k B_i $ is also open in the tropical Zariski topology, consider the following polynomials
\begin{equation*}
\ns{f}_i(x) = \begin{cases}
(x_i+\epsilon x_i +a_i)(x_i+b_i + \epsilon b_i)  & \text{if } B_i =  (a_i, b_i) \\
x_i+\epsilon x_i +c_i  & \text{if } B_i = (c_i, \infty]
\end{cases}
\end{equation*}
with corresponding tropical varieties 
\begin{equation*}
V(\ns{f}_i) = \begin{cases}
(-\infty, a_i] \cup [b_i, \infty]  & \text{if } B_i =  (a_i, b_i) \\
(-\infty, c_i]  & \text{if } B_i = (c_i, \infty]
\end{cases}
\end{equation*} 

Since $\bigcup_{i=1}^k V(\ns{f}_i) = B^c$ is closed in the tropical topology we conclude that $B$ is open.  This proves that the tropical Zariski topology is finer than the standard topology. \vspace{0.1cm}

On the other hand, if $U \subseteq \f{T}^k$ be open in our tropical topology, then $U^c$ is closed and one can write $U^c = V(\ns{I})$ for some ideal $\ns{I}$ in $\ta{\f{T}}[x]$.  Pick a collection of generators for $\ns{I}$, say $\ns{I} = \langle \ns{f}_i \rangle_{i \in \Lambda}$.  According to \Cref{rmrken}, each $V(\ns{f}_i)$ is a finite union of closed polyhedra (some of these polyhedra could have dimension less than $k$), therefore, each $V(\ns{f}_i)$ is closed in the standard topology.  Hence, by \Cref{newprop}\ref{newintersection}, $V(\ns{I}) = \bigcap_{i \in \Lambda} V(\ns{f}_i)$ is closed and $U$ is open in the standard topology. 
\end{proof}

The essential property of tropical Zariski topology, that follows from \Cref{etztopo}, is that the induced topology on tropical variety $V$ is independent of the choice of embedding into an affine tropical space.  On the other hand, in \cite{GG2}, they defined \textit{strong Zariski topology} on $\f{T}^k$ by using the $\f{T}$-points of arbitrary subschemes as the closed sets and they show this topology is exactly the Euclidean topology (see lemma 3.4.4 in \cite{GG2}).  Their idea is that, in strong Zariski topology, closed subsets of $\f{T}^k$ are defined by congruence varieties, \textit{i.e.}, equations of the form $f = g$ where $f, g \in \f{T}[x]$.  Combining this with \Cref{Enhanced} one can deduce \Cref{etztopo} too. \vspace{0.1cm}

Let $X = \p{A}$ be an affine variety over valued field $K$.\footnote{We assume $K$ is an algebraically closed field which is complete with respect to some nontrivial valuation $v : K \to \f{T}$.  This implies that the value group $v(K)$ is dense in $\f{T}$.} If $i : X \hookrightarrow \f{A}^m$ is an embedding into an affine space, then there is natural tropicalization $\textbf{Trop}(X, i) \subseteq \f{T}^m$ of $X$ with respect to this embedding (see \cite{SP} for the details).  These tropicalizations, equipped with the Euclidean topology induced from $\f{T}^m$, form an inverse system in a natural way.  Recall that the Berkovich space $X^{\text{an}}$ of $X$ is the set of all semi-valuations $v_x$ on $A$ extending the valuation on $K$ with the weakest topology such that for each $a \in A$ the evaluation map $X^{\text{an}} \to \f{T} \colon v_x \mapsto v_x(a)$ is continuous.  By theorem of Payne, see \cite{SP} Theorem 1.1, the natural map from $X^{\text{an}}$ to the inverse limit $\varprojlim_i \textbf{Trop}(X, i)$ is homeomorphism. In \cite{GG2} they show that one can obtain the Berkovich space $X^{\text{an}}$ by a single tropicalization with respect to embedding into an infinite dimensional affine space which is called the universal tropicalization of $X$, \textit{i.e.}, $X^{\text{an}} \cong \textbf{Trop}_{\text{univ}}(X)$.

\section{Tropical Varieties and Congruence Varieties} \label{Tropical Varieties and Congruence Varieties}

In this section we show congruence varieties could be thought as (classical) tropical varieties living in higher dimensions.  In fact, Let $i: \f{T}^k \to \f{T}^{k+1}$ be the horizontal embedding $i(a_1, a_2, ...,a_k) = (a_1, a_2, ..., a_k, 0)$ and let $E$ be a congruence of $\f{T}[x_1, ..., x_k]$.  We shall construct ideal $I$ in $\f{T}[x_1, ..., x_k, y]$ such that $i(V(E))=V(I)$.

\begin{remark}
This construction is far from being a corresponding between congruences and ideals.  We do not even associate to congruence $E$ (respectively, congruence variety $V(E)$) a specific ideal $I$ (respectively, tropical variety $V(I)$) with above property.  See \Cref{Ex14,Ex15}.
\end{remark}

\begin{ex} \label{Ex14}
Let $E = \langle x^2 + x + 1 \sim x \rangle$ and consider the congruence variety $V(E) = [0, 1]$.  Now let $I = \langle x^2+xy+x+1y, y+0 \rangle$ and $J = \langle x^2 + xy + 2y^2 + x + 1y + 2, y+0 \rangle$ be ideals in $\f{T}[x, y]$.  Then $i(V(E)) = V(I) = V(J)$.
\end{ex}

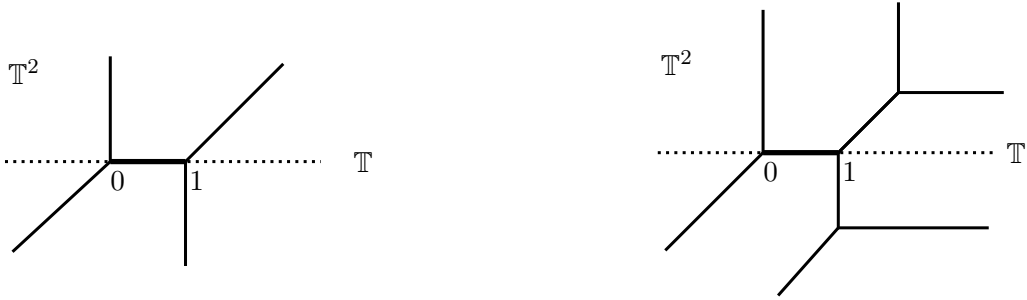
\begin{figure}[H]
\captionsetup{font=footnotesize}
\begin{center}
\begin{subfigure}[t]{.44\textwidth}
\centering
\begin{tikzpicture}
% Here we draw left picture:
\draw [line width=0.4mm](-2.3,-1.2) -- (-1,0) -- (-1,1.4);
\draw [line width=0.4mm](-1,0) -- (0,0) -- (1.3, 1.3);
\draw [line width=0.4mm](0,-1.4) -- (0,0);
\draw [line width=0.4mm, dotted](-2.4,0) -- (1.8,0);
\draw [line width=0.7mm](-1,0) -- (0, 0);
%\draw (0, -1.5) node [below]{$.$};
\fill[white] (-0.1, -1.8) rectangle (0.1, -1.4);
\draw (-1.13,0) node [below right]{$0$};
\draw (-0.08,0) node [below right]{$1$};
\draw (2.1, 0) node [right]{$\f{T}$};
\draw (-1.8, 0.9) node [above left]{$\f{T}^2$};
\end{tikzpicture}
\end{subfigure}  %Don't leave blank line
\hfill
\begin{subfigure}[t]{.44\textwidth}
\centering
\begin{tikzpicture}
% Here we draw right picture:
\draw [line width=0.4mm](-2.3,-1.3) -- (-1,0) -- (-1,1.9);
\draw [line width=0.4mm](-1,0) -- (0,0) -- (0.8,0.8) -- (0.8,2);
\draw [line width=0.4mm](0,-1) -- (0,0) -- (0.8,0.8) -- (2.2,0.8);
\draw [line width=0.4mm](-0.8,-1.9) -- (0,-1) -- (2,-1);
\draw [line width=0.4mm, dotted](-2.4,0) -- (2.1,0);
\draw [line width=0.7mm](-1,0) -- (0, 0);
\draw (-1.13,0) node [below right]{$0$};
\draw (-0.08,0) node [below right]{$1$};
\draw (2.1, 0) node [right]{$\f{T}$};
\draw (-1.8, 0.9) node [above left]{$\f{T}^2$};
\end{tikzpicture}
\end{subfigure}
\caption{If one embeds congruence variety $V(E)=[0,1]$ into $\f{T}^2$, then it could be identified with tropical varieties in different ways.}
\label{proj}
\end{center}
\end{figure}

\begin{lemma}\label{lem1}
For congruence $E=\langle x^n +cx^m \sim cx^m \rangle$ of $\f{T}[x_1, x_2, ..., x_k]$ there is an ideal $I$ in $\f{T}[x_1, x_2, ..., x_k, y]$ such that $i\left(V(E)\right)=V(I)$.
\end{lemma}
\begin{proof}
Let $I=\langle y(x^n +cx^m)+cx^m , y+0 \rangle$.  Then the nonlinear locus of $I$ is the intersection of nonlinear loci of $y(x^n +cx^m)+cx^m$ and $y+0$.  That is all points $(a_1, a_2, ..., a_k, 0) \in \f{T}^{k+1}$ where $\min \{ \sum n_ia_i, c+\sum m_ia_i, c+\sum m_ia_i \}$ attains the minimum in at least two terms which is exactly the half-space $\{ (a,0) : c+\sum m_ia_i \le \sum n_ia_i \}$.  This completes the proof since $V(E)$ is the half-space $\{ a : c+\sum m_ia_i \le \sum n_ia_i \}$.
\end{proof}

Recall that a \textit{convex} subset of $\f{T}^k$ is just an intersection of (possibly infinitely many) half-spaces.  We are allowing intersections of infinitely many half-spaces since congruences and ideals of $\f{T}[x]$ are not necessarily finitely generated, see \Cref{infgen}.  If one prefers to work with finitely generated congruences and ideals then, in the following, they could replace "convex subset" with "polyhedron". 

\begin{lemma}\label{lem2}
If $P$ is a closed convex subset of $\f{T}^k$, then there is an ideal $I$ in $\f{T}[x_1, x_2, ..., x_k, y]$ such that $i\left(P\right)=V(I)$.
\end{lemma}
\begin{proof}
The subset $P$ can be written as an intersection of half-spaces $P=\bigcap_j H_j$.  One applies \Cref{lem1} to each half-space $H_j$ and gets ideal $I_j$ in $\f{T}[x_1, x_2, ..., x_k, y]$ such that $i(H_j) = V(I_j)$.  Then
\begin{align*}
i(P) = i \Big( \bigcap_j H_j \Big)= \bigcap_j i \big( H_j \big)= \bigcap_j V \big(I_j\big) = V \Big( \sum_j I_j \Big)
\end{align*} 
where the last equality follows from \Cref{knownprop}\ref{knownintersection}.  In fact, with above notations, if $I_j = \langle f_j , y+0 \rangle$ where $j \in \Lambda$, then $I = \langle f_j, y+0 \rangle_{j \in \Lambda}$.
\end{proof}

\begin{lemma} \label{poly}
If $X \subseteq \f{T}^k$ is a finite union of convex subsets of $\f{T}^k$ then $i\left(X\right)=V(I)$ for some ideal $I$ in $\f{T}[x_1, x_2, ..., x_k, y]$.
\end{lemma}
\begin{proof}
Let $X=\bigcup_{j=1}^r P_j$.  For every $0 \le j \le r$, by \Cref{lem2}, there is an ideal $I_j$ of $\f{T}[x_1, x_2, ..., x_k, y]$ such that $i(P_j)=V(I_j)$.  Then
\begin{align*}
i(X) = i \Big( \bigcup_{j=1}^r P_j \Big) = \bigcup_{j=1}^r i \big(P_j\big) = \bigcup_{j=1}^r V\big(I_j\big) = V\big(I_1 ... I_r\big)
\end{align*}
where the last equality holds by \Cref{knownprop}\ref{knownunion}.  
\end{proof}

\begin{prop}\label{polypro}
Let $E$ be a congruence of $\f{T}[x]$.  Then $i(V(E))=V(I)$ for some ideal $I$ in $\f{T}[x_1, x_2, ..., x_k, y]$.
\end{prop}
\begin{proof}
Assume $E=\langle f_j \sim g_j \rangle_{j \in \Lambda}$.  For each $j \in \Lambda$ the congruence variety of $E_j = \langle f_j \sim g_j \rangle$ is
\begin{align*}
V \big(E_j\big) = \big\{ a \in \f{T}^k : f_j(a) = g_j(a) \big\}
\end{align*}
which is a finite union of polyhedra in $\f{T}^k$ since $f_j$ and $g_j$ are polynomials in $\f{T}[x]$.  So, the congruence variety $V(E)=\bigcap_{j \in \Lambda} V(E_j)$ is a finite union of convex subsets of $\f{T}^k$.  Then \Cref{poly} guarantees the existence of an ideal $I$ in $\f{T}[x_1, x_2, ..., x_k, y]$ such that $i(V(E))=V(I)$.
\end{proof}

\begin{remark} 
For given congruence $E$, in order to construct ideal $I$ as in \Cref{polypro}, one mostly needs at first to construct a new congruence $E'$ with generators as in \Cref{lem1} such that $V(E')=V(E)$.  In the following, we explain it by an example.
\end{remark}

\begin{ex} \label{Ex15}
Let $E = \langle x_1+x_2+0 \sim 0 \rangle$.  Then $V(E) = \{ (a_1, a_2) \colon a_1, a_2 \geq 0 \}$ is the first quadrant in $\f{T}^2$ which can be written as the intersection of two half-spaces: $V(E) = H_1 \cap H_2$ where $H_i = \{(a_1, a_2) \colon a_i \geq 0 \}$ for $i=1, 2$.  The half-space $H_i$ is defined by $x_i \geq 0$ so we consider linear relation $x_i + 0 \sim 0$.  Then we apply \Cref{lem1} to $E'=\langle x_1 + 0 \sim 0, x_2 + 0 \sim 0 \rangle$ and get $I = \langle x_1+y+0, x_2+y+0, y+0  \rangle$.  We are not allowed to apply \Cref{lem1} to $E = \langle x_1+x_2+0 \sim 0 \rangle$ since $x_1+x_2+0 \sim 0$ is not in the form of $x^n +cx^m \sim cx^m$.  Note that if we do so, then we get $J = \langle x_1+x_2+y+0, y+0  \rangle$ and $V(J) = \{ (a_1, a_2) \colon a_1, a_2 \geq 0 \} \cup \{ (a_1, a_2) \colon a_1=a_2 \leq 0 \}$.
\end{ex}

\begin{figure}[H]
\captionsetup{font=footnotesize}
\begin{center}
\begin{subfigure}[t]{.44\textwidth}
\centering
\begin{tikzpicture}
% Here we draw left picture:
\fill[gray!80] (0, 0) rectangle (1.8, 1.8);
\draw [thick, ->] (0,0) -- (2,0) node [below] {$x$};
\draw [thick, ->] (0,0) -- (0,2) node [left] {$y$};
\fill[white] (-0.1, -0.1) rectangle (1.8, -1.2);
\end{tikzpicture}
\caption{$V(\langle x_1+y+0, x_2+y+0, y+0 \rangle)$.}
\label{ex15fig:a}
\end{subfigure}  %Don't leave blank line
\hfill
\begin{subfigure}[t]{.44\textwidth}
\centering
\begin{tikzpicture}
% Here we draw left picture:
\fill[gray!80] (0, 0) rectangle (1.8, 1.8);
\draw [thick, ->] (0,0) -- (2,0) node [below] {$x$};
\draw [thick, ->] (0,0) -- (0,2) node [left] {$y$};
\draw [thick] (0,0) -- (-1.2,-1.2);
\end{tikzpicture}
\caption{$V(J)$ where $J = \langle x_1+x_2+y+0, y+0  \rangle$.}
\label{ex15fig:b}
\end{subfigure}
\label{ex15fig}
\end{center}
\end{figure}

We devote the rest of this section to show that for given congruence $E$ in $\f{T}[x]$, instead of allowing an extra variable $y$ and constructing ideal $I$ in $\f{T}[x, y]$, we can work over the semiring of tropical dual numbers $\ta{\f{T}}$ and construct an ideal $\ns{I}$ in $\ta{\f{T}}[x]$ such that $V(E) = V(\ns{I})$.  Our construction in this case is also far from providing a functor from the category of congruences in $\f{T}[x]$ to the category of ideals over $\ta{\f{T}}$.

\begin{prop} \label{Enhanced}
Let $E$ be a congruence of $\f{T}[x]$.  Then there is an ideal $\ns{I}$ in $\ta{\f{T}}[x]$ such that $V(E) = V(\ns{I})$. 
\end{prop}
\begin{proof}
Since $V(E)$ is finite union of convex subsets, by \Cref{newprop}\ref{newunion}, one can assume $V(E)$ is just a convex subset of $\f{T}^k$.  Because a convex subset is intersection of half-spaces, by applying \Cref{newprop}\ref{newintersection}, it is sufficient to prove the proposition for the case that $V(E)$ is a half-space. \vspace{0.1cm}

Consider the congruence $E' = \langle x^n +cx^m \sim cx^m \rangle$ such that $V(E)=V(E')$, in other words, assume $V(E)$ is the half-space $\{ a \in \f{T}^k : c+\sum m_ia_i \le \sum n_ia_i \}$.  Now for $\ns{I} = \langle x^n +(c + c\epsilon)x^m \rangle$ one can promptly verify that the tropical variety $V(\ns{I})$ is the same half-space $\{ a \in \f{T}^k : c+\sum m_ia_i \le \sum n_ia_i \}$.
\end{proof}

\begin{cor} \label{Enhancedclassical}
Let $\ns{I}$ be an ideal of $\ta{\f{T}}[x]$.  Then there is an ideal $I$ in $\f{T}[x, y]$ such that $i(V(\ns{I})) = V(I)$.
\end{cor}

For congruence $E$ in \Cref{Ex15}, the tropical variety associated to $\ns{I} = \langle x_1 +(0 + \epsilon), x_2 +(0 + \epsilon) \rangle$ is the first quadrant.

\end{document}